\documentclass{amsart}

\usepackage[all]{xy}
\usepackage{amssymb}
\usepackage{amsthm}
\usepackage[headings]{fullpage}

\newcommand{\QQ}{\mathbb Q}
\newcommand{\ZZ}{\mathbb Z}
\newcommand{\CC}{\mathbb C}
\newcommand{\FF}{\mathbb F}
\newcommand{\Qp}{\QQ_p}
\newcommand{\Zp}{\ZZ_p}
\newcommand{\calH}{\mathcal{H}}
\newcommand{\calF}{\mathcal{F}}
\newcommand{\calL}{\mathcal{L}}
\newcommand{\calO}{\mathcal{O}}
\newcommand{\calG}{\mathcal{G}}
\newcommand{\ff}{\mathfrak{f}}
\newcommand{\pp}{\mathfrak{p}}

\newcommand{\Hom}{\mathrm{Hom_{cts}}}

\DeclareMathOperator{\Res}{Res}
\DeclareMathOperator{\Sel}{Sel}
\DeclareMathOperator{\Gal}{Gal}
\DeclareMathOperator{\ord}{ord}
\DeclareMathOperator{\Tr}{Tr}
\DeclareMathOperator{\GL}{GL}
\DeclareMathOperator{\Irr}{Irr}
\DeclareMathOperator{\HOM}{Maps}
\DeclareMathOperator{\Det}{Det}
\newcommand{\HH}{\mathcal{H}}
\newcommand{\RR}{\mathbb R}
\newcommand{\MGH}{\mathfrak{M}_{\calH}(\calG)}
\DeclareMathOperator{\Ind}{Ind}

\newtheorem{theorem}{Theorem}[section]
\newtheorem{proposition}[theorem]{Proposition}
\newtheorem{lemma}[theorem]{Lemma}

\newtheorem{remark}[theorem]{Remark}
\newtheorem{definition}[theorem]{Definition}
\newtheorem{conjecture}[theorem]{Conjecture}
\newtheorem{condition}[theorem]{Condition}
\newtheorem*{assumption}{Technical assumption}

\begin{document}

\title{Non-commutative $p$-adic $L$-functions for supersingular primes}

\author{Antonio Lei}
\address{Department of Mathematics and Statistics\\
Burnside Hall\\
McGill University\\
Montreal QC\\
Canada H3A 0B9}
\email{antonio.lei@mcgill.ca}
\thanks{The author is supported by a CRM-ISM fellowship.}

\begin{abstract}
Let $E/\QQ$ be an elliptic curve with good supersingular reduction at $p$ with $a_p(E)=0$. We give a conjecture on the existence of analytic plus and minus $p$-adic $L$-functions of $E$ over the $\Zp$-cyclotomic extension of a finite Galois extension of $\QQ$ where $p$ is unramified. Under some technical conditions, we adopt the method of Bouganis and Venjakob for $p$-ordinary CM elliptic curves to construct such functions for a particular non-abelian extension. 
\end{abstract}


\maketitle

 
\section{Introduction}
Let $E/\mathbb{Q}$ be an elliptic curve with good ordinary reduction at an odd prime $p$. Coates {\it et al.} \cite{CKFVS,coateshowson} have recently developed the framework of the Iwasawa theory of $E$ over a $p$-adic Lie extension $\calF$ of $\QQ$ that contains the $\Zp$-cyclotomic extension $\QQ_\infty$. Let $\calG=\Gal(\calF/\QQ)$ and $\Gamma=\Gal(\QQ_\infty/\QQ)$. It is predicted that there exists a $p$-adic $L$-function $\calL_{\calG,E}\in K_1(\Lambda_{\calO}(\calG)_{S^*})$, where $\Lambda_{\calO}(\calG)_{S^*}$ is an appropriate localisation of an Iwasawa algebra of $\calG$ and $\calL_{\calG,E}$ interpolates the complex $L$-values of $E$ in the following sense. For all Artin representations $\rho$ of $\calG$, $\calL_{\calG,E}$ is expected to satisfy
\begin{equation}\label{eq:ord}
\calL_{\calG,E}(\rho)=\frac{e_p(\rho)}{u^{\ff_p(\rho)}}\times\frac{P_p(\rho,u^{-1})}{P_p(\rho^\vee,w^{-1})}\times\frac{L_R(E,\rho^\vee,1)}{\Omega_+^{d^+(\rho)}\Omega_-^{d^-(\rho)}}.
\end{equation}
Here, $e_p$ denotes the local epsilon factor at $p$, $u$ and $w$ are respectively the $p$-adic unit and non-unit roots of the quadratic $X^2-a_p(E)X+p$, $P_p(\star,X)$ is the polynomial describing the Euler factor of $\star$ at $p$ and $R$ is the set consisting of the prime $p$ and all primes at which $E$ has multiplicative reduction. We shall review some of the notation in \S\ref{sec:no}. The main conjecture predicts that such a $p$-adic $L$-function, should it exist, generates the characteristic ideal of the dual Selmer group $X(E/\calF)$, which is conjectured to lie in the $\MGH$-category (see \S\ref{sec:ore} below).

The existence of $\calL_{\calG,E}$ would generalise the work of Mazur and Swinnerton-Dyer \cite{MSD74}, where they constructed an element $\calL_{\Gamma,E}\in\Lambda(\Gamma)$ that interpolates the $L$-values of $E$ twisted by finite characters of $\Gamma$. More precisely, if $\chi$ is a character on $\Gamma$ with conductor $p^n>1$, then
\[
\calL_{\Gamma,E}(\chi)=\frac{\tau(\chi)}{u^n}\times\frac{L(E,\chi^{-1},1)}{\Omega_+},
\]
where $\tau(\chi)$ denotes the Gauss sum of $\chi$.

When $E$ has good supersingular reduction, no such element exists. Instead, Amice and V{\'e}lu \cite{amicevelu75} constructed two admissible $p$-adic $L$-functions $\calL_{E,\Gamma}^u$ and $\calL_{E,\Gamma}^w$, one for each of the two roots, $u$ and $w$, to $X^2-a_p(E)X+p$. Even though they have similar interpolating properties as their counterpart in the ordinary case, they do not lie in $\Lambda(\Gamma)$. Pollack \cite{pollack03} resolved this  by decomposing $\calL_{E,\Gamma}^u$ and $\calL_{E,\Gamma}^w$ into linear combinations of two elements $\calL^\pm_{E,\Gamma}\in\Lambda(\Gamma)$ when $a_p(E)=0$ and it has been generalised to the case where $a_p(E)\ne0$ by Sprung \cite{sprung09}. We shall concentrate on the former case in this paper. Pollack's $p$-adic $L$-functions exhibit the following interpolating properties. Let $\chi$ be a character on $\Gamma$ of conductor $p^n>1$. If $n$ is even,
\begin{equation}\label{eq:+inter}
\calL_{\Gamma,E}^+(\chi)=\frac{\tau(\chi)}{\omega^+(\chi)}\times\frac{L(E,\chi^{-1},1)}{\Omega_+},
\end{equation}
whereas if $n$ is odd,
\begin{equation}\label{eq:-inter}
\calL_{\Gamma,E}^-(\chi)=\frac{\tau(\chi)}{\omega^-(\chi)}\times\frac{L(E,\chi^{-1},1)}{\Omega_+}.
\end{equation}
Here, $\omega^+(\chi)$ and $\omega^-(\chi)$ are some non-zero factors that come from the plus and minus logarithms $\log^\pm$ defined in \cite{pollack03}. It is shown in {\it op. cit.} that $\calL_{\Gamma,E}^\pm$ are uniquely determined by \eqref{eq:+inter} and \eqref{eq:-inter} respectively. We shall review some of the details of Pollack's work in \S\ref{sec:pol}.

The main goal of this paper is to formulate a conjecture on the existence of two elements $\calL_{\calG,E}^\pm\in K_1(\Lambda_{\calO}(\calG)_{S^*})$ with interpolating properties similar to \eqref{eq:+inter} and \eqref{eq:-inter} for $\calG=\Gal(\calF/\QQ)$, where $\calF$ is the $\Zp$-cyclotomic extension of a finite Galois extension of $\QQ$ in which $p$ is unramified. Let $\rho$ be an irreducible Artin representation of $\calG$ with $\ff_p(\rho)=n$. We define in \S\ref{sec:con} the factors $\omega^+(\rho)$ and $\omega^-(\rho)$, depending on the parity of $n$. We then go on to predict that there exist $\calL_{\calG,E}^\pm\in K_1(\Lambda_{\calO}(\calG)_{S^*})$ that satisfy
\begin{equation}\label{eq:++inter}
\calL_{\calG,E}^+(\rho)=\frac{e_p(\rho)}{\omega^+(\rho)}\times\frac{L_R(E,\rho^\vee,1)}{\Omega_+^{d^+(\rho)}\Omega_-^{d^-(\rho)}}
\end{equation}
if $n$ is even, whereas for odd $n$,
\begin{equation}\label{eq:--inter}
\calL_{\calG,E}^-(\rho)=\frac{e_p(\rho)}{\omega^-(\rho)}\times\frac{L_R(E,\rho^\vee,1)}{\Omega_+^{d^+(\rho)}\Omega_-^{d^-(\rho)}}.
\end{equation}
Here $R$ is the set of primes as defined in the ordinary case above.
Note that the Euler factors at $p$ are trivial for the representations which we consider, which explains why they are not present in our conjectural formulae. Roughly speaking, $\calL_{\calG,E}^+$ and $\calL_{\calG,E}^-$ are  each  interpolating the $L$-values of $E$ twisted by  ``half" of all irreducible Artin representations of $\calG$. Despite these seemingly weaker properties, we show that, as in the ordinary case, if $\calL_{\calG,E}^\pm$ exist, they are uniquely determined by \eqref{eq:++inter} and \eqref{eq:--inter} respectively as elements of $K_1(\Lambda_{\calO}(\calG)_{S^*})$ modulo the kernel of a determinant map.

Kobayashi \cite{kobayashi03} defined the plus and minus Selmer groups $\Sel^\pm_p(E/\QQ_\infty)$ for $E$ over $\QQ_\infty$ and formulate a main conjecture predicting that the characteristic ideals of their Pontryagin duals $X^\pm(E/\QQ_\infty)$ should be generated by $\calL_{\Gamma,E}^\pm$ (this has been proved by Pollack and Rubin \cite{pollackrubin04} when $E$ has complex multiplication). In \cite{leizerbes11}, we have generalised Kobayashi's construction to arbitrary $p$-adic Lie extensions. Therefore, analogous to the ordinary case, the existence of $\calL_{\calG,E}^\pm$ would allow us to formulate a main conjecture for these plus and minus Selmer groups, which relates the characteristic ideals of $X^\pm(E/\calF)$ to $\calL_{\calG,E}^\pm$. See \S\ref{sec:main} for details.

In \cite{bv10}, Bouganis and Venjakob proved the existence of the $p$-adic $L$-function $\calL_{\calG,E}\in K_1(\Lambda_{\calO}(\calG)_{S^*})$ that satisfies \eqref{eq:ord}, where $E$ is an elliptic curve with complex multiplication that has good ordinary reduction at $p$ and $\calG=\Gal(\QQ(E[p^\infty])/\QQ)$ assuming that the Selmer group satisfies the $\MGH$-conjecture. Their construction makes use of the $2$-variable $p$-adic $L$-function of Yager \cite{yager82}, which does not exists in the supersingular case. We have nonetheless managed to adopt their method, together with some of the ideas from \cite{kimparkzhang}, to construct two elements $\calL_{\calG,E}^\pm$ that satisfy \eqref{eq:++inter} and \eqref{eq:--inter} respectively under some technical assumptions when $E$ is an elliptic curve with complex multiplication by $K$ which has good supersingular reduction at $p$ and $\calG$ is the Galois group of the extension of $\QQ$ by $\QQ_\infty\cdot F$, where $F$ is an abelian extension of $K$ in which $p$ is unramified. Note that $\calG$ could be abelian, but there do exist examples for which this is not the case, e.g. $K=\QQ(\sqrt{-3})$ and $F=K(\sqrt[3]{2})$. The details of our construction are given in \S\ref{sec:special}.

\section*{Acknowledgement}
The idea of this paper was originally developed when the author was a member of Monash University. He is extremely grateful for many helpful discussions with Daniel Delbourgo and Lloyd Peters on the subject while he was there. He is also indebted to David Loeffler for pointing out a few mistakes in an earlier version of this paper and to Henri Darmon and Alex Bartel for answering many of his questions. Finally, the author would like to thank the anonymous referee for suggestions that help improving the paper.

\section{Notation and setup}\label{sec:no}

\subsection{$p$-adic Lie extension} \label{sec:paLe}
  Throughout this paper, $p$ is an odd prime. If $K$ is a field of characteristic $0$, either local or global, $G_K$ denotes its absolute Galois group, $\kappa$ the $p$-cyclotomic character on $G_K$ and $\calO_K$ the ring of integers of $K$. We write $c$ for the complex conjugation in $G_\QQ$.

We let $K_\infty$ denote the $\Zp$-cyclotomic extension of $K$. The Galois group $\Gal(K_\infty/K)$ is written as $\Gamma_K$. When $K=\QQ$, we simply write $\Gamma$ for $\Gamma_\QQ$. We fix a topological generator $\gamma$ of $\Gamma$.

 From now on, we fix a finite Galois extension $F$ of $\QQ$ with Galois group $\calH$. We assume that $p$ is unramified in $F$. Let $\calF=F_\infty$. Our assumption on $F$ implies that $F\cap\QQ_\infty=\QQ$. Hence, $F_\infty/\QQ$ is Galois. We write $\calG$ for its Galois group. Then
 \begin{equation}\label{eq:direct}
 \calG\cong\calH\times\Gamma,
 \end{equation}
 which allows us to view both $\Gamma=\Gamma_\QQ$ and $\calH$ as subgroups of $\calG$ (e.g. $\calG$ could be the Galois group studied in \cite{hara}, where $\calH$ is the group of $4\times4$ upper-triangular unipotent matrices over $\FF_p$). We fix a family of subgroups of $\calH$, written as $\{U_i\}_{i\in I}$, such that the following is true.
 \begin{condition}\label{hyp:family}
 For all irreducible representations $\rho$ of $\calH$, the trace of $\rho$ is equal to a $\ZZ$-linear combinaton of representations of the form $\Ind_{U_i}^\calH\chi_{\rho,i}$ where $i\in I$ and $\chi_{\rho,i}$ is an one-dimensional character on $U_i$.
 \end{condition}
\begin{remark}Note that such a family exists by Brauer's theorem on induced characters.\end{remark}

By \eqref{eq:direct}, for all irreducible representations $\rho$ of $\calG$, there exists an one-dimensional character $\rho_0$ of $\Gamma$ such that the trace of $\rho$ is equal to a $\ZZ$-linear combinaton of representations of the form 
\[
\Ind_{U_i\times\Gamma}^\calG(\chi_{\rho,i}\otimes\rho_0)\cong\left(\Ind_{U_i}^\calG\chi_{\rho,i}\right)\otimes\rho_0
\]
 where $i\in I$ and $\chi_{\rho,i}$ is an one-dimensional character on $U_i$.
  
  We write $V_i=[U_i,U_i]$ for all $i\in I$.

\subsection{Iwasawa algebras and power series}
 Given a finite extension $K$ of $\Qp$ and a $p$-adic Lie group $G$, $\Lambda_{\calO_K}(G)$ denotes the Iwasawa algebra of $G$ over $\calO_K$, i.e.
 \[
  \varprojlim_{N}\calO_K[G/N], 
 \] 
 where the inverse limit runs over the open normal subgroups of $G$. When $K=\Qp$ (so $\calO_K=\Zp$), we suppress $\Zp$ from the notation and write $\Lambda(G)$ for $\Lambda_{\Zp}(G)$. We denote $\Lambda_{\calO_K}(G)\otimes_{\calO_K} K$ by $\Lambda_K(G)$.

Let $r\in\RR_{\ge0}$. We define
\[
\HH_r=\left\{\sum_{n\geq0}c_{n} X^n\in\CC_p[[X]]:\sup_{n}\frac{|c_{n}|_p}{n^r}<\infty\ \right\},
\]
where $|\cdot|_p$ is the $p$-adic norm on $\CC_p$ such that $|p|_p=p^{-1}$. We write $\HH_r(\Gamma)=\{f(\gamma-1):f\in\HH_r\}$. In other words, the elements of $\HH_r$ (respectively $\HH_r(\Gamma)$) are the power series in $X$ (respectively $\gamma-1$) over $\CC_p$ with growth rate $O(\log_p^r)$. 

Given a subfield $K$ of $\CC_p$, we write $\HH_{r,K}=\HH_r\cap K[[X]]$ and similarly for $\HH_{r,K}(\Gamma)$. In particular, $\HH_{0,K}(\Gamma)=\Lambda_{K}(\Gamma)$.

If $h=\sum_{n\ge 0} c_{n}(\gamma-1)^n\in\HH_r(\Gamma)$ and $\lambda\in\Hom(\Gamma,\CC_p^\times)$, we write
\[
h(\lambda)=\sum_{n\ge 0} c_{n}(\lambda(\gamma)-1)^n\in\CC_p.
\]
More generally, if $\rho$ is an Artin representation on $\Gamma$, then $\rho$ may be decomposed into a finite sum of characters on $\Gamma$, say 
\[
\rho\cong\bigoplus_{i=1}^r\chi_i.
\]
 We then write
\[
h(\rho)=\prod_{i=1}^rh(\chi_i).
\]


\subsection{Artin representations}

Let $V$ be a finite dimensional vector space over $\CC$. If $\rho:G_{\QQ}\rightarrow \GL_{\CC}(V)$ is an Artin representation, we write $d(\rho)$ (respectively $d^\pm(\rho)$) for the dimension of $V$ (respectively the subspace of $V$ on which the complex conjugation acts as $\pm1$). The contragredient representation of $\rho$ is denoted by $\rho^\vee$.

We now review the definition of the local epsilon factor $e_p(\rho)=e_p(\rho|_{G_{\Qp}})$ at $p$ (see \cite[\S3]{tatebackground} and \cite[\S6.10]{dokchitsers07} for details). It depends on a choice of Haar measure and an additive character on $\Qp$. In this article, we take the canonical measure $\mu$ on $\Qp$ with $\mu(\Zp)=1$ and the additive character that sends $ap^{-n}$ to $e^{2a\pi i/p^n}$ for $a\in\Zp$. It is multiplicative in the sense that $e_p(\rho_1\oplus\rho_2)=e_p(\rho_1)e_p(\rho_2)$, so it is enough to define $e_p(\rho)$ for irreducible $\rho$.

If $\rho$ is one-dimensional, we may factor $\rho|_{G_{\Qp}}$ into $\rho_0\rho'$, where $\rho_0$ is an unramified character and $\rho'$ is a Dirichlet character of conductor $p^n$. In this case, the epsilon factor is defined to be
\[
e_p(\rho)=\rho_0(p)^n\times \tau(\rho'),
\]
where $\tau(\rho')$ denotes the Gauss sum of $\rho'$. If $\rho=\Ind_{F}^\QQ(\chi)$ for some character $\chi$ on $G_F$, where $F$ is a number field, we may define $e_{\mathfrak{p}}(\chi)$ for each place $\mathfrak{p}$ of $F$ above $p$ in a similar way. The epsilon factor of $\rho$ at $p$ is then defined to be
\[
e_p(\rho)=e_p(\Ind_{F}^\QQ(\mathbf{1}))\prod_{\mathfrak{p}|p}\frac{e_{\mathfrak{p}}(\chi)}{e_{\mathfrak{p}}(\mathbf{1})}.
\]


\subsection{A canonical Ore set}\label{sec:ore}

We recall the definition of an Ore set $S$ from \cite{CKFVS}. Let $\calO$ be the ring of integer of a finite extension of $\Qp$. If $\calG$ and $\calH$ are as in \S\ref{sec:paLe}, we define
\[
S(\calO)=\{f\in\Lambda_\calO(\calG):\text{$\Lambda_{\calO}(\calG)/\Lambda_{\calO}(\calG)f$ is a finitely generated $\Lambda_{\calO}(\calH)$-module}\}
\]
and $S(\calO)^*=\cup_{n\ge 0}p^nS(\calO)$. We write $\Lambda_{\calO}(\calG)_{S^*}$ for the localisation of $\Lambda_{\calO}(\calG)$ at $S(\calO)^*$.

Let $U_i$ be the subgroups as fixed in Section~\ref{sec:paLe}. For each $i\in I$, we have a natural map
\[
\theta_i:K_1(\Lambda_\calO(\calG)_{S^*})\rightarrow K_1(\Lambda_\calO(U_i\times\Gamma)_{S^*})\rightarrow K_1(\Lambda_\calO(U_i/V_i\times\Gamma)_{S^*})=\Lambda_{\calO}(U_i/V_i\times\Gamma)_{S^*}^\times.
\]

Let $\Irr_p(\calG)$ be the additive group generated by isomorphic classes of $p$-adic representations of $\calG$. There is a determinant map
\begin{eqnarray*}
\Det:K_1(\Lambda_\calO(\calG)_{S^*})&\rightarrow&\HOM\left(\Irr_p(\calG),\overline{\Qp}\cup\{\infty\}\right)\\
x&\mapsto&(\rho\mapsto x(\rho)).
\end{eqnarray*}

We write $\MGH$ for the category of all finitely generated $\Lambda(\calG)$-modules which are $S^*(\Zp)$-torsion. There is a connecting map
\[
\partial_{\calG}:K_1(\Lambda(\calG)_{S^*})\rightarrow K_0(\MGH).
\]
It is surjective when $\calG$ has no $p$-torsion. Given an element $M$ in $\MGH$, a characteristic element for $M$ is any $\xi_M\in K_1(\Lambda(\calG)_{S^*})$ such that $\partial_{\calG}(\xi_M)=[M]$.


\section{Conjectures}

\subsection{Pollack's plus and minus $p$-adic $L$-functions}\label{sec:pol}

We first recall Pollack's construction of plus and minus $p$-adic $L$-functions for the $\Zp$-cyclotomic extension of $\QQ$ in \cite{pollack03}. The plus and minus logarithms are defined as follows.

\begin{definition}
Let $r\ge0$ be an integer, define
\begin{eqnarray*}
\log^+_r&=&\prod_{s=0}^{r}\frac{1}{p}\prod_{n=1}^\infty\frac{\Phi_{2n}(\kappa(\gamma)^{-s}\gamma)}{p},\\
\log^-_r&=&\prod_{s=0}^{r}\frac{1}{p}\prod_{n=1}^\infty\frac{\Phi_{2n-1}(\kappa(\gamma)^{-s}\gamma)}{p},
\end{eqnarray*}
where $\Phi_m$ denotes the $p^m$-cyclotomic polynomial.
\end{definition}

\begin{lemma}\label{lem:pollackzero}
Let $\lambda$ be a character on $\Gamma$. Then $\log^+_r(\lambda)=0$ if and only if $\lambda=\kappa^s\theta$, where $s\in[0,r]$ and $\theta$ is a Dirichlet character whose conductor is an odd power of $p$. Similarly, $\log^-_r(\lambda)=0$ if and only if $\lambda=\kappa^s\theta$, where $s\in[0,r]$ and $\theta$ is a Dirichlet character whose conductor is an even power of $p$.
\end{lemma}
\begin{proof}
\cite[Lemma~4.1]{pollack03}.
\end{proof}

\begin{remark}
Both $\log^+_r$ and $\log^-_r$ are elements of $\calH_{r/2,\Qp}(\Gamma)$ as given by \cite[Lemma~4.5]{pollack03}.
\end{remark}

 Let $f=\sum a_nq^n$ be a normalised eigen-newform of weight $k\ge2$, level $N$ and Nebentypus character $\epsilon$. Throughout, we assume $p\nmid N$. We write $F_p(f)$ for the completion of the coefficient field $\QQ(a_n:n\ge1)$ at a prime above $p$. Let $\alpha$ be a root of $X^2-a_pX+\epsilon(p)p^{k-1}=0$ with $h:=\ord_p(\alpha)<k-1$. Amice and V\'{e}lu \cite{amicevelu75} constructed $L_{f,\alpha}\in\HH_{h,F_p(f)}(\Gamma)$ (written as $\calL^\alpha_{\Gamma,E}$ in the introduction) with the following interpolating properties. For all integers $j\in[0,k-2]$ and Dirichlet characters $\theta$ of conductor $p^n$, 
\begin{equation}\label{eq:AV}
L_{p,\alpha}(\theta\kappa^j)=e_\alpha(\theta,j)\times\frac{p^{n(j+1)}j!}{\tau(\theta^{-1})(-2\pi i)^j}\times\frac{L(f,\theta^{-1},j+1)}{\Omega_f^\delta},
\end{equation}
where $\delta=(-1)^j$, $\tau(\theta^{-1})$ denotes the Gauss sum of $\theta^{-1}$,
\[
e_\alpha(\theta,j)=\frac{1}{\alpha^n}\left(1-\frac{\theta^{-1}(p)\epsilon(p)p^{k-2-j}}{\alpha}\right)\left(1-\frac{\theta(p)p^j}{\alpha}\right)
\]
and $\Omega_f^\pm$ are some choice of periods.
Moreover, $L_{f,\alpha}$ is uniquely determined by its values at $\theta\kappa^j$ as given by \eqref{eq:AV}.

\begin{remark}\label{rk:even}
We are only considering the $\Zp$-cyclotomic extension here, rather than the whole extension by $p$-power roots of unity as studied in \cite{amicevelu75}. In particular, $\theta(-1)$ is always $1$.
\end{remark}

\begin{theorem}\label{thm:pollack}
Let $f$ and $\alpha$ be as above with $a_p=0$. There exist $L_f^\pm\in\Lambda_{F_p(f)}(\Gamma)$ (written as $\calL_{\Gamma,E}^\pm$ in the introduction) such that
\[
L_{f,\alpha}=\log_{k-1}^+L_f^++\alpha\log_{k-1}^-L_f^-.
\]
\end{theorem}
\begin{proof}
This is the main result of \cite{pollack03}.
\end{proof}

In particular, if $\alpha_1$ and $\alpha_2$ are the two roots to $X^2+\epsilon(p)p^{k-1}=0$, we have $\alpha_1=-\alpha_2$ and
\begin{eqnarray}
L_{f}^+&=&\frac{L_{p,\alpha_1}+L_{p,\alpha_2}}{2\log_{k-1}^+},\label{eq:pol1}\\
L_{f}^-&=&\frac{L_{p,\alpha_2}-L_{p,\alpha_1}}{(\alpha_2-\alpha_1)\log_{k-1}^-}\label{eq:pol2}.
\end{eqnarray}
Therefore, we can readily combine \eqref{eq:AV} with \eqref{eq:pol1} and \eqref{eq:pol2} to obtain the interpolating formulae of $L_{f}^+$ and $L_f^-$ respectively as given below.

\begin{lemma}\label{lem:inter}
Let $j\in[0,k-2]$ be an integer and $\theta$ a Dirichlet conductor $p^n>1$. Write $\delta$ as in \eqref{eq:AV}. If $n$ is even, then
\[
L_f^+(\theta\kappa^j)=\frac{1}{(-\epsilon(p)p^{k-1})^{n/2}\log_{k-1}^+(\theta\chi^j)}\times\frac{p^{n(j+1)}j!}{\tau(\theta^{-1})(-2\pi i)^j}\times\frac{L(f,\theta^{-1},j+1)}{\Omega_f^\delta},
\]
whereas $\theta\kappa^j(L_f^+)=0$ if $n=1$. For all odd $n$, we have
\[
L_f^-(\theta\kappa^j)=\frac{1}{(-\epsilon(p)p^{k-1})^{(n+1)/2}\log_{k-1}^-(\theta\chi^j)}\times\frac{p^{n(j+1)}j!}{\tau(\theta^{-1})(-2\pi i)^j}\times\frac{L(f,\theta^{-1},j+1)}{\Omega_f^\delta}.
\]
Moreover,
\begin{eqnarray*}
L_f^+(\kappa^j)&=&\frac{1-p^{-1}}{\log_{k-1}^+(\chi^j)}\times\frac{j!}{(-2\pi i)^j}\times\frac{L(f,j+1)}{\Omega_f^\delta},\\
L_f^-(\kappa^j)&=&\frac{p^{-j-1}+\epsilon(p)^{-1}p^{j-k+1}}{\log_{k-1}^-(\chi^j)}\times\frac{j!}{(-2\pi i)^j}\times\frac{L(f,j+1)}{\Omega_f^\delta}.
\end{eqnarray*}
\end{lemma}

\begin{definition}\label{def:polfactor}
Let $\theta$ be a Dirichlet character of conductor $p^n$. For even $n$, define 
\[
\omega^+(\theta)=(-p)^{n/2}\log_{1}^+(\theta).
\]
For odd $n$, define
\[
\omega^-(\theta)=(-p)^{(n+1)/2}\log_{1}^-(\theta).
\]
\end{definition}
For a weight $2$ modular form, we have the following simplified version of Lemma~\ref{lem:inter}.
\begin{lemma}\label{lem:inter2}
Assume that $f$ is of weight $2$. Let $\theta$  be a Dirichlet conductor $p^n>1$.  If $n$ is even, then
\[
L_f^+(\theta)=\frac{\tau(\theta)}{\epsilon(p)^{n/2}\omega^+(\theta)}\times\frac{L(f,\theta^{-1},1)}{\Omega_f^+},
\]
whereas $L_f^+(\theta)=0$ if $n=1$. For all odd $n$, we have
\[
L_f^-(\theta)=\frac{\tau(\theta)}{\epsilon(p)^{(n+1)/2}\omega^-(\theta)}\times\frac{L(f,\theta^{-1},1)}{\Omega_f^+}.
\]
Moreover,
\begin{eqnarray*}
L_f^+(\mathbf{1})&=&(p-1)\times\frac{L(f,1)}{\Omega_f^+},\\
L_f^-(\mathbf{1})&=&(1+\epsilon(p)^{-1})\times\frac{L(f,1)}{\Omega_f^+}.
\end{eqnarray*}
\end{lemma}
\begin{proof}
Recall from Remark~\ref{rk:even} that $\theta(-1)=1$, so we have $\tau(\theta)\tau(\theta^{-1})=\theta(-1)p^n=p^n$. The result is then immediate from Lemma~\ref{lem:inter}. See also \cite[(3.4)-(3.7)]{kobayashi03}.
\end{proof}


\subsection{Conjectural analytic $p$-adic $L$-functions}
\label{sec:con}
In this section, we fix an elliptic curve $E$ defined over $\QQ$. We assume that $E$ has good supersingular reduction at $p$ with $a_p(E)=0$. As in the introduction, we write $R$ for the set consisting of the prime $p$ and the primes where $E$ has multiplicative reduction. Let $\Omega_+$ and $\Omega_-$ denote the real and complex periods of $E$ respectively. Let $\calF$ and $\calG$ be as defined in Section~\ref{sec:paLe}.

\begin{definition}\label{def:con}
Let $\rho$ be an Artin representation on $\calG$. We say that $\rho$ is of even (respectively odd) $\Gamma$-conductor if, via  \eqref{eq:direct}, 
\[
\rho\cong \rho_0\otimes\rho'
\]
 for some representation $\rho_0$ of $\calH$ and some one-dimensional character $\rho'\ne\mathbf{1}$ of $\Gamma$ whose conductor is an even (respectively odd) power of $p$.
\end{definition}

\begin{remark}\label{rk:irred}
By \eqref{eq:direct}, an irreducible representation $\rho$ of $\calG$ is of the form $\rho_0\otimes\rho'$ where $\rho_0$ is an irreducible representation of $\calH$ and $\rho'$ is an one-dimensional character of $\Gamma$. In particular, if $\rho'\ne\mathbf{1}$, it has either even or odd $\Gamma$-conductor. Moreover, all Artin representations of $\calG$ that have even (respectively odd) $\Gamma$-conductors are direct sums of such irreducible representations.
\end{remark}

\begin{lemma}\label{lem:nonzero}
Let $\rho$ be an Artin representation on $\calG$. If $\rho$ is of even $\Gamma$-conductor, then
\[
\omega^+(\rho|_\Gamma)=\omega^+(\rho')^{d(V)}\ne0.
\]
Similarly, if $\rho$ is of odd $\Gamma$-conductor,
\[
\omega^-(\rho|_\Gamma)=\omega^-(\rho')^{d(V)}\ne0.
\]
\end{lemma}
\begin{proof}
If $\rho$ is of even or odd $\Gamma$-conductor, we have
\[
\rho|_\Gamma\cong(\rho')^{\oplus d(V)}.
\]
By definition, $\omega^\pm(\rho')$ and $\log_1^\pm(\rho')$ differ by a power of $p$. Hence the result by Lemma~\ref{lem:pollackzero}.
\end{proof}

As in Definition~\ref{def:polfactor}, we make the following definition to simplify our notation.

\begin{definition}
Let $\rho$ be an Artin representation on $\calG$ that is of even $\Gamma$-conductor, we write
\[
\omega^+(\rho)=\omega^+(\rho|_\Gamma),
\]
which is non-zero by Lemma~\ref{lem:nonzero}. If on the other hand $\rho$ is of odd $\Gamma$-conductor, we write
\[
\omega^-(\rho)=\omega^-(\rho|_\Gamma),
\]
which again is non-zero by Lemma~\ref{lem:nonzero}.
\end{definition}

 We now formulate our conjecture on the existence of plus and minus $p$-adic $L$-functions of $E$ over $\calG$.

\begin{conjecture}\label{conj:exist}
There exist two elements $\calL_{\calG,E}^\pm\in K_1(\Lambda_\calO(\calG)_{S^*})$ for the ring of integer $\calO$ of some finite unramified extension of $\Qp$ such that $\calL_{\calG,E}^\nu(\rho)\ne\infty$ and
\begin{equation}\label{eq:conj}
\calL_{\calG,E}^\nu(\rho)=\frac{e_p(\rho)}{\omega^\nu(\rho)}\times\frac{L_R(E,\rho^\vee,1)}{\Omega_+^{d^+(\rho)}\Omega_-^{d^-(\rho)}}
\end{equation}
for all Artin representations $\rho$ on $\calG$ that has even (respectively odd) $\Gamma$-conductor with $\nu=+$ (respectively $\nu=-$).
\end{conjecture}

This is analogous to \cite[Conjecture~5.7]{CKFVS}. 
\begin{remark}\label{rk:sameirred}
Remark~\ref{rk:irred} implies that we might assume that the representations considered in the statement of Conjecture~\ref{conj:exist} are irreducible.
\end{remark}

We now show that despite having much weaker interpolating properties than their counterpart in the good ordinary case, \eqref{eq:conj} uniquely determines $\calL_{\calG,E}^\pm$ modulo $\ker(\Det)$.

\begin{theorem}
If Conjecture~\ref{conj:exist} holds, $\calL_{\calG,E}^\pm$ are uniquely determined by \eqref{eq:conj} modulo the kernel of $Det$.
\end{theorem}
\begin{proof}

 For all $\xi\in K_1(\Lambda_\calO(\calG)_{S^*})$ and monomial representations $\rho$ on $\calG$, with $\rho=\Ind_{U_i\times\Gamma}^{\calG}\chi_\rho$, we have
\[
\Det(\xi)(\rho)=\xi(\rho)=\theta_i(\xi)(\chi_\rho).
\]
Hence, by Condition~\ref{hyp:family}, $\Det(\xi)$ is uniquely determined by its image under the map $\prod_{i\in I}{\theta_i}(\xi)$. In other words,
\[
\ker(\Det)=\ker\left(\prod_{i\in I}{\theta_i}\right).
\]
Therefore, it suffices to show that \eqref{eq:conj} uniquely determines an element in 
\[
K_1(\Lambda_\calO(\calG)_{S^*})/\ker\left(\prod_{i\in I}{\theta_i}\right).
\]

 A character $\chi$ on $U_i/V_i\times\Gamma$  decomposes into $\chi_0\otimes\chi'$, where $\chi_0=\chi|_{U_i/V_i}$ and $\chi'=\chi|_{\Gamma}$. Therefore, by Weierstrass' preparation theorem, an element in $\Lambda_\calO(U_i/V_i\times\Gamma)^\times_{S^*}$ is uniquely determined by its values at characters of the form $\chi_0\otimes\chi'$ for all $\chi_0$ and an infinite number of $\chi'$. Artin representations $\rho$ induced from characters on $U_i\times\Gamma$ that have even (or odd) $\Gamma$-conductor provide such a set of characters $\chi_\rho$ because Definition~\ref{def:con} does not impose any restrictions on $\chi_0$ and $\chi'$ can send $\gamma$ to an infinite number of primitive roots of unity. The values predicted by \eqref{eq:conj} therefore uniquely determine $\theta_i(\calL_{\calG,E}^\pm)$ for all $i\in I$,  hence the result.

\end{proof}

\subsection{Main conjecture}\label{sec:main}

In \cite{leizerbes11}, we have defined the signed Selmer groups $\Sel^\pm_p(E/\calF)$, which are subgroups of the usual Selmer group $\Sel_p(E/\calF)$. Let $X^\pm(E/\calF)$ be their respective Pontryagin duals. We conjecture that the following holds.

\begin{conjecture}\label{con:belong}
The dual Selmer groups $X^\pm(E/\calF)$ belong to the category $\MGH$.
\end{conjecture}

This corresponds to \cite[Conjecture~5.1]{CKFVS}. If Conjecture~\ref{con:belong} holds and $\calG$ contains no $p$-torsion, then there exist characteristic elements $\xi_{X^\pm(E/\calF)}\in K_1(\Lambda(\calG)_{S^*})$ as discussed in \S~\ref{sec:ore}. We may then formulate a main conjecture that relates these characteristic elements to the conjectural analytic $p$-adic $L$-functions predicted by Conjecture~\ref{conj:exist}.

\begin{conjecture}[Main conjecture]\label{conj:main}
Let $i:K_1(\Lambda(\calG)_{S^*})\rightarrow K_1(\Lambda_{\calO}(\calG)_{S^*})$ be the natural homomorphism. Assume that $\calG$ does not contain an element of order $p$ and both Conjectures~\ref{conj:exist} and~\ref{con:belong} hold. Then 
\[
\calL_{\calG,E}^\pm\equiv i(\xi_{X^\pm(E/\calF)})\mod i(K_1(\Lambda_{\calO}(\calG))).
\]
\end{conjecture}
This is analogous to \cite[Conjecture~5.8]{CKFVS}.

\section{A special case}\label{sec:special}

In this section, we assume that our elliptic curve $E$ has complex multiplication by $\calO_K$, where $K$ is an imaginary quadratic extension of $\QQ$. If $\psi$ is a character on $G_K$, we write $\psi^c$ for the character that sends $g$ to $\psi(cgc^{-1})$.

We recall that a Grossencharacter of $K$ is simply a continuous homomorphism $\phi:C_K\rightarrow\CC^\times$, where $C_K$ is the idele class group of $K$. It has complex $L$-function
\[
L(\phi,s)=\prod_v(1-\phi(v)N(v)^{-s})^{-1},
\]
where the product runs through the finite places $v$ of $K$ at which $\phi$ is unramified, $\phi(v)$ is the image of the uniformiser of $K_{v}$ under $\phi$ and $N(v)$ is the norm of $v$. We say that $\phi$ is of type $(m,n)$ for some $m,n\in\ZZ$ if the restriction of $\phi$ to the archimedean part $\CC^\times$ of $C_K$ is of the form $z\mapsto z^m\bar{z}^n$. By Class Field Theory, we may equally view $\phi$ as a character on the Galois group $G_K$.

Since we assume that $E$ has complex multiplication, we have $R=\{p\}$. Moreover, there exist a  Grossencharacter  $\phi$ be of type $(-1,0)$ over $K$ and a weight $2$ modular form $f_{\phi}$ such that
\[
L(E,s)=L(\phi,s)=L(f_\phi,s).
\]
Then $L(E/K,s)=L(\phi,s)L(\phi^c,s)$. We continue to assume that $E$ has good supersingular reduction at $p$. This implies that $p$ is inert in $K$ and $a_p(E)$ is automatically $0$.

Let $F$ be a finite abelian extension of $K$ in which $p$ is unramified. Then, $K_\infty\cap F=K$. Moreover, $F_\infty/K$ is abelian and
\[
\Gal(F_\infty/K)\cong\Gal(F/K)\times\Gal(K_\infty/K).
\]

Let $A=\Gal(F/K)$, $G=\Gal(F_\infty/K)$ and $\Delta=\Gal(K/\QQ)$. We further assume that $F_\infty/\QQ$ is Galois and write $\calG=\Gal(F_\infty/\QQ)$. Then
\begin{equation}\label{eq:gsemi}
\calG\cong\Delta\ltimes G.
\end{equation}
As remarked in the introduction, $\calG$ can be either abelian or non-abelian, depending on whether $\Delta$ acts on $G$ trivially.

For the rest of this section, we shall show under some technical conditions that there exist two elements $\calL_{\calG,E}^\pm\in K_1(\Lambda_{\calO_{F_A}}(\calG)_{S^*})$ satisfying the interpolating properties predicted by Conjecture~\ref{conj:exist} for some finite extension $F_A$ of $\Qp$. Let us briefly outline our strategy here. In \cite{bv10}, Bouganis and Venjakob constructed a non-commutative $p$-adic $L$-function over the extension $\QQ(E[p^\infty])/\QQ$ for a $p$-ordinary elliptic curve with complex multiplication by $\calO_K$. They made use of the $2$-variable $p$-adic $L$-function of Yager \cite{yager82} for the extension $\QQ(E[p^\infty])/K$ and consider its image in $K_1(\Lambda_{\calO}(\Gal(\QQ(E[p^\infty])/\QQ)))$ under a natural map, which turns out to satisfy \eqref{eq:ord} up to some correction factor $\calL_\Omega$. In the supersingular case, a corresponding $2$-variable $p$-adic $L$-function has not yet been constructed. We instead construct plus and minus $p$-adic $L$-functions for the extension $F_\infty/K$ using ideas from \cite{kimparkzhang}. Once this is done, we can apply the machineries developed by Bouganis and Venjakob to construct our desired elements for the extension $F_\infty/\QQ$.


\subsection{Analytic $p$-adic $L$-functions}

Let $\eta$ be an one-dimensional character on $A$. We write $F_\eta$ for the field $\Qp(\eta(g):g\in A)$. Note that $\phi\eta$ is once again a Grossencharacter of type $(-1,0)$ over $K$. Hence, by \cite[Theorem~3.4]{ribet77}, there exists a CM modular form $f_{\phi\eta}$ such that $L(f_{\phi\eta},s)=L(\phi\eta,s)$. Its conductor is coprime to $p$ by our assumption on $F$, so the plus and minus $p$-adic $L$-functions $L^\pm_{f_{\phi\eta}}\in\Lambda_{F_\eta}(\Gamma)$ exist, as given by Theorem~\ref{thm:pollack}. Note that the periods $\Omega^\pm_{f_{\phi\eta}}$ are not unique, but we may choose $\Omega^+_{f_{\phi\eta}}$ to be $\Omega_+$ for all $\eta$ by the following lemma.

\begin{lemma}
For any choice of $\Omega^+_{f_{\phi\eta}}$, there exists a constant $C\in F_\eta^\times$ such that $C\Omega_+=\Omega^+_{f_{\phi\eta}}$.
\end{lemma}
\begin{proof}
Given a modular form $f$, let $L_f$ be the $\QQ$-vector space generated by the $L$-values $L(f,\chi,1)$ where $\chi$ runs through the set of Dirichlet characters.

By the algebraicity property of the periods, the lemma would follow from the inclusion
\[
L_{f_{\phi\eta}}\subset F_\eta\cdot L_{f_\phi},
\]
which is a consequence of the Fourier inversion formula for $\eta$ and Birch's lemma (c.f. \cite[\S4]{haran}).
\end{proof}

Under our choice of periods, we define the following.

\begin{definition}
Let $\nu\in\{+,-\}$ and $\delta=0$ if $\nu=+$, whereas $\delta=1/2$ if $\nu=-$. Define
\[
L^\nu_{G,\phi}=\sum_{\eta\in\hat{A}}e_\eta\cdot \overline{\eta}(p)^{\delta} L^\nu_{f_{\phi\bar{\eta}}}\in \Lambda_{F_A}(\Gamma)[A],
\]
where $F_A=\cup_{\eta\in\hat{A}}F_\eta(\eta(p)^{1/2})$ and $e_\eta$ is the idempotent $|A|^{-1}\sum_{g\in A} \bar{\eta}(g)g$.
\end{definition}

 If we identify $\Gamma_K$ with $\Gamma$, we can treat $L^\pm_{G,\phi}$ as elements of $\Lambda_{F_A}(A\times\Gamma_K)\cong\Lambda_{F_A}(G)$. We shall do so from now on without further notice.

\begin{lemma}\label{lem:intercom}
Let $\chi$ be a one-dimensional character on $G$ with $\chi_0=\chi|_{A}$ and $\chi'=\chi|_{\Gamma_K}$ with conductor $p^n>1$, then
\[
L^\nu_{G,\phi}(\chi)=\frac{\chi_0(p)^{n/2}\tau(\chi')}{\omega^\nu(\chi')}\times \frac{L(\phi,\overline{\chi},1)}{\Omega_+}
\]
where $\nu=+$ if $n$ is even and $\nu=-$ if $n$ is odd.
\end{lemma}
\begin{proof}
For any $\eta\in\hat{A}$, $\chi(e_\eta)=0$ if $\chi_0\ne\eta$ and $\chi(e_\eta)=1$ otherwise. This implies
\[
L^\nu_{G,\phi}(\chi)= \overline{\chi_0}(p)^{\delta}L^\nu_{f_{\phi\overline{\chi_0}}}(\chi').
\]
Since $f_{\phi\overline{\chi_0}}$ is of weight $2$ and its Nebentypus character takes value $\overline{\chi_0}(p)$ at $p$, Lemma~\ref{lem:inter2} implies that
\begin{eqnarray*}
 L^\nu_{f_{\phi\overline{\chi_0}}}(\chi')&=&\overline{\chi_0}(p)^{\delta}\times\frac{\tau(\chi')}{\overline{\chi_0}(p)^{\lfloor (n+1)/2\rfloor}\omega^\nu(\chi')}\times\frac{L(\phi\overline{\chi_0},\overline{\chi'},1)}{\Omega_\phi^+}\\
 &=&\frac{{\chi_0}(p)^{\lfloor (n+1)/2\rfloor-\delta}\tau(\chi')}{\omega^\nu(\chi')}\times\frac{L(\phi\overline{\chi_0},\overline{\chi'},1)}{\Omega_+}
\end{eqnarray*}
for the appropriate parity of $n$.
But $\lfloor(n+1)/2\rfloor -\delta =n/2$, so we are done.
\end{proof}

Let us recall from \cite[Corollary~5.11]{pollack03} and \cite[Lemma~6.5]{lei09} that $L^\pm_{f_{\phi\eta}}\ne0$ for all $\eta\in\hat{A}$. It therefore makes sense to talk about the $\mu$-invariants of $L^\pm_{f_{\phi\eta}}$. We from now on assume that the following is true.

\begin{assumption}
Under our choice of periods, the $\mu$-invariant of $L^\pm_{f_{\phi\eta}}$ is independent of $\eta\in \hat{A}$ (but may depend on the sign $\pm$). 
\end{assumption}

\begin{lemma}\label{lem:mgh}
Under our technical assumption, $L^\nu_{G,\phi}\in\Lambda_{\calO_{F_A}}(G)_{S^*}^\times= K_1(\Lambda_{\calO_{F_A}}(G)_{S^*})$.
\end{lemma}
\begin{proof}
Since we assume that $L^\pm_{f_{\phi\eta}}$ have the same $\mu$-invariant, there exists an integer $n$ (depends on the sign $\pm$) such that
\[
\Lambda_{\calO_{F_A}}(G)/\Lambda_{\calO_{F_A}}(G)p^nL^\pm_{f_{\phi\eta}}
\]
is finitely generated over $\Zp$, and hence over $\calO_{F_A}[A]$. In particular,
\[
p^nL^\pm_{f_{\phi\eta}}\in S(\calO_{F_A})^*,
\]
which proves the lemma.
\end{proof}

Let $\iota$ be the inclusion map
\[
\iota:\Lambda_{\calO_{F_A}}(G)_{S^*})\rightarrow \Lambda_{\calO_{F_A}}(\calG)_{S^*}.
\]
This induces a map
\[
\iota_*:K_1(\Lambda_{\calO_{F_A}}(G)_{S^*})\rightarrow K_1(\Lambda_{\calO_{F_A}}(\calG)_{S^*}).
\]
Thanks to Lemma~\ref{lem:mgh}, we may now make the following definition.

\begin{definition}
Under the notation above, we define
\[
L_{\calG,\phi}^\pm:=(\iota)_*\left(L^\pm_{G,\phi}\right)\in K_1(\Lambda_{\calO_{F_A}}(\calG)_{S^*}).
\]
\end{definition}

The following lemma allows us to relate the interpolating properties of $L^\pm_{G,\phi}$ to those of $L_{\calG,\phi}^\pm$.

\begin{lemma}\label{lem:sameres}
For all Artin representations $\rho$ of $\calG$, one has
\[
L_{\calG,\phi}^\pm(\rho)=L^\pm_{G,\phi}\left(\Res^{\calG}_G\rho\right).
\]
\end{lemma}
\begin{proof}
This follows from the same proof as \cite[Lemma~2.9]{bv10}.
\end{proof}

\begin{proposition}\label{prop:paLf}
Let $\rho$ be an irreducible Artin representation on $\calG$. Then
\[
L_{\calG,\phi}^\nu(\rho)=\frac{e_p(\rho)}{\omega^\nu(\rho)}\times\frac{L(E,\rho^\vee,1)}{\Omega_+^{d(\rho)}},
\]
where $\nu=+$ (respectively $\nu=-$) if the $\rho$ has even (respectively odd) $\Gamma$-conductor.
\end{proposition}
\begin{proof}
By \cite[\S8.2]{serre78} and \eqref{eq:gsemi}, all irreducible Artin representations $\rho$ of $\calG$ are either one-dimensional or isomorphic $\Ind_G^{\calG}\chi$ for some one-dimensional character $\chi$ on $G$ with $\chi\ne\chi^c$ (the latter case occurs if and only if $\calG$ is non-abelian).

If $\rho$ is one-dimensional, write $\chi=\Res_G^\calG(\rho)$ and we factor $\chi$ into $\chi_0\otimes\chi'$ as in Lemma~\ref{lem:intercom}. We have $\chi'=\rho|_{\Gamma}$. Let $p^n$ be the conductor of $\chi'$. Via the identification of $\Gamma$ with $\Gamma_K$, we see that $\rho$ has even (respectively odd) $\Gamma$-conductor if and only if $n$ is even (respectively odd). Therefore, we can combine Lemmas~\ref{lem:intercom} and~\ref{lem:sameres} to deduce that
\begin{equation}\label{eq:interes}
L_{\calG,\phi}^\nu(\rho)=\frac{\chi_0(p)^{n/2}\tau(\chi')}{\omega^\nu(\chi')}\times \frac{L(\phi,\overline{\chi},1)}{\Omega_+}.
\end{equation}

By Frobenius reciprocity, we have
\begin{equation}\label{eq:frob}
L(\phi,\overline{\chi},1)=L(E,\overline{\rho},1)=L(E,\rho^\vee,1).
\end{equation}
Let $\rho_0$ be the unramified part of $\rho|_{G_{\Qp}}$. Then, $e_p(\rho)=\rho_0(p)^n\tau(\chi')$ by definition. But $K_p/\Qp$ is an unramified extension of degree $2$, so $\rho_0(p)^2=\chi_0(p)$. Hence, $e_p(\rho)=\chi_0(p)^{n/2}\tau(\chi')$. Together with \eqref{eq:interes} and \eqref{eq:frob}, we can therefore conclude that
\[
L_{\calG,\phi}^\nu(\rho)=\frac{e_p(\rho)}{\omega^\nu(\rho)}\times\frac{L(E,\rho^\vee,1)}{\Omega_+}
\] 
as required.

We now study the case when $\rho$ is $2$-dimensional, should it occur. In this case, we have $\rho=\Ind_G^{\calG}\chi$ for some character $\chi$ on $G$. Then $\Res^{\calG}_G\rho=\chi\oplus\chi^c$. As before, we decompose $\chi$ and $\chi^c$ into $\chi_0\otimes\chi'$ and $\chi_0^c\otimes(\chi^c)'$ respectively. As we assume that $\rho$ has even or odd $\Gamma$-conductor, we have $\chi'=(\chi^c)'\ne\mathbf{1}$. In particular,
\begin{equation}\label{eq:decom}
\rho|_{\Gamma}=\chi'\oplus(\chi^c)'=(\chi')^{\oplus2}.
\end{equation}

 We can once again apply Lemmas~\ref{lem:intercom} and~\ref{lem:sameres} to deduce that
\begin{eqnarray*}
L_{\calG,\phi}^\nu(\rho)&=&L^\nu_{G,\phi}(\chi)L^\nu_{G,\phi}(\chi^c)\\
&=&(\chi_0\chi_0^c(p))^{n/2}\times\left(\frac{\tau(\chi')}{\omega^\nu(\chi')}\right)^2\times \frac{L(\phi,\overline{\chi},1)L(\phi,\overline{\chi^c},1)}{\Omega_+^2},
\end{eqnarray*}
where $\nu$ depends on the parity of the $\Gamma$-conductor of $\rho$ as in the statement of the proposition.

As in the one-dimensional case, Frobenius reciprocity implies that
\[
L(\phi,\overline{\chi},1)L(\phi,\overline{\chi^c},1)=L(\phi,\overline{\chi}\oplus\overline{\chi^c},1)=L(E,\rho^\vee,1).
\]
By our assumption on the non-ramification of $p$ in $F$, $\rho|_{G_{\Qp}}$ in fact decomposes into the direct sum of two one-dimensional characters, say $\rho_1$ and $\rho_2$. On restricting to $G_{K_p}$, they coincide with $\chi$ and $\chi^c$. Hence, as above, we may deduce that $e_p(\rho_1)$ and $e_p(\rho_2)$ are given by $\chi_0^{n/2}(p)\tau(\chi')$ and $\chi_0^c(p)^{n/2}\tau(\chi')$. This implies that
\[
e_p(\rho)=e_p(\rho_1)e_p(\rho_2)=\chi_0\chi_0^c(p)^{n/2}(\tau(\chi'))^2.
\]
Putting all these together, we conclude that
\begin{eqnarray*}
L_{\calG,\phi}^\nu(\rho)&=&\frac{e_p(\rho)}{(\omega^\nu(\chi'))^2}\times\frac{L(E,\rho^\vee,1)}{\Omega_+^2}\\
&=&\frac{e_p(\rho)}{\omega^\nu(\rho)}\times\frac{L(E,\rho^\vee,1)}{\Omega_+^2},
\end{eqnarray*}
where the second equality follows from \eqref{eq:decom}.
\end{proof}

We can now define our plus and minus $p$-adic $L$-functions of $E$ over $\calG$.
\begin{definition}
Let $\calL_\Omega=\frac{1+c}{2}+\frac{\Omega_+}{\Omega_-}\times \frac{1-c}{2}\in\Lambda(\calG)^\times$ (c.f. \cite[Lemma~2.10]{bv10}). We define
\[
\calL_{\calG,E}^\pm:=L_{\calG,\phi}^\pm\calL_{\Omega}\in K_1(\Lambda_{\calO_{F_A}}(\calG)_{S^*}).
\]
\end{definition}

\begin{theorem}\label{thm:exist}
Let $\rho$ be an irreducible Artin representation of $\calG$. Then
\[
\calL_{\calG,E}^\nu(\rho)=\frac{e_p(\rho)}{\omega^\nu(\rho)}\times\frac{L(E,\rho^\vee,1)}{\Omega_+^{d^+(\rho)}\Omega_-^{d^-(\rho)}},
\]
where $\nu=+$ (respectively $\nu=-$) if the $\rho$ has even (respectively odd) $\Gamma$-conductor.
\end{theorem}
\begin{proof}
If $\rho$ is $1$-dimensional, then $d^+(\rho)=1$ and $\calL_\Omega(\rho)=1$. Hence, by Proposition~\ref{prop:paLf},
\begin{eqnarray*}
\calL_{\calG,E}^\nu(\rho)&=&L_{\calG,\phi}^\nu(\rho)\\
&=&\frac{e_p(\rho)}{\omega^\nu(\rho)}\times\frac{L(E,\rho^\vee,1)}{\Omega_+^{d(\rho)}},
\end{eqnarray*}
where $\nu$ corresponds to the parity of the $\Gamma$-conductor of $\rho$ as usual. If $\rho$ is $2$-dimensional, then $d^+(\rho)=d^-(\rho)=1$ and $\calL_\Omega(\rho)=\Omega_+/\Omega_-$. Therefore, we deduce from Proposition~\ref{prop:paLf} that
\begin{eqnarray*}
\calL_{\calG,E}^\nu(\rho)&=&L_{\calG,\phi}^\nu(\rho)\\
&=&\frac{e_p(\rho)}{\omega^\nu(\rho)}\times\frac{L(E,\rho^\vee,1)}{\Omega_+^{2}}\times\frac{\Omega_+}{\Omega_-}\\
&=&\frac{e_p(\rho)}{\omega^\nu(\rho)}\times\frac{L(E,\rho^\vee,1)}{\Omega_+\Omega_-}.
\end{eqnarray*}
\end{proof}

\begin{remark}
As remarked above, $R=\{p\}$ because $E$ has complex multiplication. Since the Euler factor of $L(E,\rho^\vee,1)$ at $p$ is trivial, Theorem~\ref{thm:exist}, together with Remark~\ref{rk:sameirred}, give an affirmative answer to Conjecture~\ref{conj:exist} in this particular setting.
\end{remark}


\subsection{Remarks on the main conjecture}
Since $p$ is unramified in $F$, $p$ splits into distinct primes $\pp_1\cdots\pp_d$ in $F$. Results in \cite[\S4]{leizerbes11}  imply that the signed Selmer groups $\Sel_p^\pm(E/F_{\infty})$ can be explicitly described as follows. For $n\ge0$, let
\[
\Sel_p^\pm(E/F_n)=\ker\left(\Sel_p(E/F_n)\rightarrow\bigoplus_{i=1}^d\frac{H^1(F_{\pp_i,n},E[p^\infty])}{E(F_{\pp_i,n})^\pm\otimes\Qp/\Zp}\right),
\]
where
\[
E(F_{\pp_i,n})^\pm=\left\{x\in E(F_{\pp_i,n}):\Tr_{n/m+1}(x)\in E(F_{\pp_i,m}) \forall m\in S_n^\pm\right\}.
\]
Here, $\Tr_{n/m+1}$ denotes the trace map $E(F_{\pp_i,n})\rightarrow E(F_{\pp_i,m+1})$ with respect to the group law on $E$ and
\begin{eqnarray*}
S_n^+&=&[0,n-1]\cap2\ZZ;\\
S_n^-&=&[0,n-1]\cap(2\ZZ+1).
\end{eqnarray*}
Then we have
\[
\Sel_p^\pm(E/F_{\infty})=\varinjlim\Sel^\pm_p(E/F_n).
\]
Note that this is the same definition as given in \cite{kimparkzhang}, where $p$ is assumed to split completely in $F$.

The main conjecture (Conjecture~\ref{conj:main}) relates our $p$-adic $L$-functions $\calL^\pm_{\calG,E}$ to the dual Selmer groups $X^\pm(E/F_\infty)$. In \cite[\S2]{bv10}, it has been shown that the main conjecture for the extension $\QQ(E[p^\infty])/\QQ$ follows from the $\MGH$-conjecture and the $2$-variable main conjecture in the ordinary case. Unfortunately, we do not have the corresponding $2$-variable main conjecture in our current setting
, so we cannot adopt the method of Bouganis and Venjakob directly here.




\bibliographystyle{amsalpha}
\bibliography{references}
\end{document}